\documentclass{amsart}
\usepackage{amsfonts}
\usepackage{graphicx}
\usepackage{xcolor}
\usepackage{tikz}

\usepackage{cite}

\newtheorem{theorem}{Theorem}
\newtheorem{proposition}[theorem]{Proposition}
\newtheorem{corollary}[theorem]{Corollary}
\newtheorem{lemma}[theorem]{Lemma}

\theoremstyle{plain}

\newtheorem{definition}{Definition}

\theoremstyle{definition}

\newtheorem{remark}{Remark}
\newcommand{\R}{{\mathbb R}}
\newcommand{\N}{{\mathbb N}}

\newcommand{\F}{{F}}
\newcommand{\cN}{{\mathcal N}}
\renewcommand{\P}{{\mathbb P}}

\newcommand{\bomega}{{\boldsymbol{\omega}}}

\DeclareMathOperator{\supp}{supp}
\DeclareMathOperator*{\esssup}{ess\,sup}
\DeclareMathOperator*{\essinf}{ess\,inf}
\newcommand{\ee}{{\mathbb E}}

\newcommand{\comment}[1]{{}}

\input{tcilatex}

\begin{document}
\title[Attainable Assouad-like dimensions]{Attainable Assouad-like dimensions
of randomly generated Moran sets and measures
}

\author{Kathryn E. Hare}
\address{Dept. of Pure Mathematics, University of Waterloo, Waterloo, Ont.,
Canada, N2L 3G1}
\email{kehare@uwaterloo.ca}
\author{Franklin Mendivil}
\address{Department of Mathematics and Statistics, Acadia University,
Wolfville, N.S., Canada, B4P 2R6}
\email{franklin.mendivil@acadiau.ca}

\thanks{The research of K. Hare is partially supported by NSERC 2016:03719.
The research of F. Mendivil is partially supported by NSERC\ 2019:05237.}



\begin{abstract}
In this paper we study the Assouad-like $\Phi$ dimensions of sets and measures 
that are constructed by a random weighted iterated function system  of similarities.
These dimensions are distinguished by the depth of the scales
considered and thus provide more refined infomation about the local
geometry/behaviour of a set or measure.
The Assouad dimensions are important well-known examples.
We determine the almost sure value of the upper and lower Assouad dimension of the random set.
We also determine the range of attainable upper and lower (small) $\Phi$ dimensions of the measures, in
both the situation where the probability weights can depend on the scaling factors and when they cannot.
In the later case we find that there is a ``gap'' between the dimension of the set and the dimensions of 
the associated family of random measures.
\end{abstract}

\maketitle


\section{Introduction}

The dimensional properties of random constructions have been a topic of
interest in fractal geometry for many years;
classic papers on the topic include \cite{Fa,Gr,Ham,MW}.
In this paper, we investigate the dimensions of
random Moran sets and measures that arise from random weighted iterated
function systems (RwIFS) and satisfy a separation condition. These are the
attractors and measures formed by randomly choosing from among a 
(possibly uncountable) collection
of iterated function systems of similarities and probabilities at each step
in the Moran construction. 
We note that these are also sometimes called one-variable fractals and
fractal measures. Early papers studying their dimensional properties include
\cite{BHS} where a formula was given for the almost sure Hausdorff dimension of
the random Moran sets and \cite{Tr17} where this value was shown to coincide with
the box dimension of the set.

In recent years there has been interest in Assouad and Assouad-like
dimensions of sets and measures (c.f. \cite{F,FH,FMT,FT,FY,GHM,HH,HM,HM2,KL,Ol,Tr17,Tr,Tr2}). 
In contrast to the Hausdorff and box dimensions,
which are global in nature, the Assouad and Assouad-like dimensions provide
information about the local `size' of the objects. In \cite{FMT} and \cite{Tr17} a
formula was given for the almost sure upper Assouad dimension of random
Moran sets. Interestingly, unlike the case for self-similar sets arising
from an IFS satisfying the open set condition, the upper Assouad dimension
does not, in general, coincide with the Hausdorff dimension of these random
attractors.

The Assouad-like $\Phi $ dimensions range between the box and Assouad
dimensions. The $\Phi$ dimensions are a family of dimensions which measure  the extremes of the local behaviour
of a set or measure by comparing two different scales, with the difference in the scales 
controlled by the choice of the function $\Phi$.
The Assouad dimensions and the $\theta$-Assouad spectrum 
are special examples. 
In \cite{Tr2} it was shown that
the upper $\Phi $ dimensions of a random Moran set arising from
equicontractive similarities almost surely coincides with the upper Assouad
dimension of the set if the dimension function 
$\Phi $ $<<\log |\log x|/|\log x|$ (`small' $\Phi $) and coincides 
with the box dimension if $\Phi >>\log |\log x|/|\log x|$ (`large' $\Phi $). 
One of the main results of this paper is to give a formula for the almost sure 
upper (and lower) Assouad dimensions of the random Moran sets and 
prove that these coincide with the upper (resp., lower) $\Phi$ dimensions of the sets 
for all small $\Phi$,
regardless of the contraction factors. 

The same threshold function and behaviour was found to apply to both the upper and lower 
$\Phi $ dimensions of random Moran measures in \cite{HM,HM2} and formulas were
given there for determining these dimensions for both the small and large $\Phi $.

Imagine that we are given a collection of iterated function systems of similarities and an
underlying random process that determines which of these IFS are chosen at each step
of the construction (so a RIFS).
Based on this, we get a family of RwIFS by using different choices of probability weights.
This produces families of random Moran measures almost all of which have the same
random Moran set as their support (the random set from the fixed RIFS). 
It is a natural question to ask about the attainable range of $\Phi$ dimensions
of these random measures.
More specifically, since the upper $\Phi $ dimension of a
measure is always at least the upper $\Phi $ dimension of its support, it is
natural to ask if the upper $\Phi $ dimension of one of these random
measures is as small as that of the supporting random set. 
In \cite{HM2} an
example was given to show that for the small $\Phi $ dimensions this need
not be the case, however it was unclear if this example was a typical or an
exceptional situation.

Here we show that for small dimension functions $\Phi $ the answer to this
question depends on the underlying random process. If the random process
does not require the probabilities to be independent of the similarities,
then almost surely the upper $\Phi $ dimension of a random Moran set agrees
with the minimal upper $\Phi $ dimension of the family of random measures it
supports. Furthermore, every greater value is attained as the upper $\Phi $
dimension of one of the random measures, except possibly the value infinity.
But if the random process requires the probabilities to be independent of
the similarities, then there is a gap between the minimal measure dimension
and the set dimension (as was the case with the example in \cite{HM2}), other
than in exceptional situations. A simple formula is given for the minimal
attainable upper $\Phi $ dimension of the family of measures and we prove
that all greater values are also attainable, except possibly infinity.

Similar statements hold for the lower $\Phi $ dimensions. Interestingly, the
maximal attainable lower $\Phi $ dimension of the random measures coincides
with the Assouad (and Hausdorff) dimensions of the deterministic, strongly
separated Moran set with contraction factors equal to the essential infimum
of the contraction factors taken over all the iterated function systems. But
this Moran set is typically not any of the random Moran sets which arise in
the random construction. 

In section 2 the reader will find basic facts about the $\Phi $ dimensions,
in section 3 the details of the random construction and in section 4 the
proofs of the main results.

\section{Assouad-like dimensions}

\label{sec:phidim}

\begin{definition}
A \textbf{dimension function} is a map $\Phi :(0,1)\rightarrow \mathbb{R}%
^{+} $ with the property that $t^{1+\Phi (t)}$ decreases to $0$ as $t$
decreases to $0$.

We will say that a dimension function $\Phi $ is \textbf{small} if 
\begin{equation*}
\Phi (t)=H(t)\frac{\log \left\vert \log t\right\vert }{\left\vert \log
t\right\vert }
\end{equation*}%
where $H(t)\rightarrow 0$ as $t\rightarrow 0$ and \textbf{large }if (with
the same notation) $H(t)\rightarrow \infty $ as $t\rightarrow 0$.
\end{definition}

Examples of small dimension functions include $\Phi (t)=0$, and the
functions $\Phi (t)=1/|\log t|^{\alpha }$ and $\Phi (t)=(\log \left\vert
\log t\right\vert )^{1/\alpha }/\left\vert \log t\right\vert $ for 
$\alpha > 1$. 
Large dimension functions include the constant functions 
$\Phi(t)=C>0$, as well as the functions $1/|\log t|^{\alpha }$ and 
$(\log\left\vert \log t\right\vert )^{1/\alpha }/\left\vert \log t\right\vert $
for $\alpha <1$.
The dichotomy ``small'' versus ''large'' corresponds to the threshold phenomena which
is observed in the almost sure values of the $\Phi$ dimensions of the random sets and measures
in our construction \cite{HM2}, and also in some other random iid constructions (see \cite{GHM2,Tr2}).

In this paper, our focus will be on the dimensions of random sets and
measures arising from small dimension functions. To recall the
definitions of these dimensions we will write$\ B(x,R)$ for the open ball
centred at $x$ belonging to the bounded metric space $X$ and radius $R,$ and
denote by $\cN_r(E)$ the least number of open balls of radius $r$ required
to cover $E\subseteq X$.

\begin{definition}
Let $\Phi $ be a dimension function.

(i) The \textbf{upper }and \textbf{lower }$\Phi $\textbf{-dimensions\ }of
the set $E\subseteq X$ are given, respectively, by 
\begin{equation*}
\overline{\dim }_{\Phi }E=\inf \left\{ 
\begin{array}{c}
d:(\exists C_{1},C_{2}>0)(\forall 0<r\leq R^{1+\Phi (R)}<R<C_{1})\text{ } \\ 
\cN_r(B(z,R)\bigcap E)\leq C_{2}\left( \frac{R}{r}\right) ^{d}\text{ }%
\forall z\in E%
\end{array}%
\right\}
\end{equation*}%
and%
\begin{equation*}
\underline{\dim }_{\Phi }E=\sup \left\{ 
\begin{array}{c}
d:(\exists C_{1},C_{2}>0)(\forall 0<r\leq R^{1+\Phi (R)}<R<C_{1})\text{ } \\ 
\cN_r(B(z,R)\bigcap E)\geq C_{2}\left( \frac{R}{r}\right) ^{d}\text{ }%
\forall z\in E%
\end{array}%
\right\} .
\end{equation*}

(ii) The \textbf{upper and lower }$\Phi $\textbf{-dimensions} of the measure 
$\mu $ on $X$ are given, respectively, by%
\begin{equation*}
\overline{\dim }_{\Phi }\mu =\inf \left\{ 
\begin{array}{c}
d:(\exists C_{1},C_{2}>0)(\forall 0<r<R^{1+\Phi (R)}\leq R\leq C_{1}) \\ 
\frac{\mu (B(x,R))}{\mu (B(x,r))}\leq C_{2}\left( \frac{R}{r}\right) ^{d}%
\text{ }\forall x\in \supp \mu 
\end{array}%
\right\} 
\end{equation*}%
and 
\begin{equation*}
\underline{\dim }_{\Phi }\mu =\sup \left\{ 
\begin{array}{c}
d:(\exists C_{1},C_{2}>0)(\forall 0<r<R^{1+\Phi (R)}\leq R\leq C_{1}) \\ 
\frac{\mu (B(x,R))}{\mu (B(x,r))}\geq C_{2}\left( \frac{R}{r}\right) ^{d}%
\text{ }\forall x\in \supp \mu 
\end{array}%
\right\} .
\end{equation*}
\end{definition}

The $\Phi $-dimensions of sets were first introduced by Fraser and Yu in 
\cite{FY} who focussed on the special case of the large dimension functions $%
\Phi _{\theta }=1/\theta -1$ for $0<\theta <1$. They called the
corresponding dimensions the upper and lower $\theta $-Assouad spectrum$.$
General $\Phi $-dimensions for sets were first extensively investigated in 
\cite{GHM}. These set dimensions motivated the $\Phi $-dimensions of
measures, introduced in \cite{HH}.

In the special case of the small dimension function $\Phi =0$, the upper and
lower $\Phi$-dimensions are known as the \textbf{upper }and \textbf{lower
Assouad dimensions} (or the Assouad dimension and lower dimension) 
and denoted $\dim _{A}\Lambda $ and $\dim _{L}\Lambda $
respectively, for sets or measures $\Lambda$. In fact, if $\Phi (t)$ is
dominated by the small dimension function $1/\left\vert \log t\right\vert $
for $t$ near $0,$ then the upper (lower) $\Phi $-dimensions coincide with
the upper (lower) Assouad dimensions of the set or measure. The Assouad
dimensions of measures are also known as the upper and lower regularity
dimensions and were studied in \cite{FH,KL, KLV}.

One obvious property of these dimensions is their ordering: If $\Phi (t)\leq
\Psi (t)$ for all $t>0$, then whether $\Lambda $ is a set or measure,%
\begin{equation*}
\overline{\dim }_{\Psi }\Lambda \leq \overline{\dim }_{\Phi }\Lambda \text{
and }\underline{\dim }_{\Phi }\Lambda \leq \underline{\dim }_{\Psi }\Lambda .
\end{equation*}%
Other well known relationships between these dimensions are stated below.
Proofs can be found in \cite{FY,GHM,HH} and the references
cited therein.

\begin{proposition}
\label{prop:dimensionproperties}
Let $E$ be a set, $\mu $ be a measure and $\Phi $ be any dimension function.
We denote by $\underline{\dim }_{B}$ ($\overline{\dim }_{B}$) the lower
(upper) box dimension and by $\dim _{H}$ the Hausdorff dimension. Then

\begin{equation*}
\dim _{L}E\leq \underline{\dim }_{\Phi }E\leq \underline{\dim }_{B}E\leq 
\overline{\dim }_{B}E\leq \overline{\dim }_{\Phi }E\leq \dim _{A}E,
\end{equation*}%
\begin{equation*}
\dim _{L}\mu \leq \underline{\dim }_{\Phi }\mu \leq \dim _{H}\mu \leq \dim
_{H}\supp\mu ,
\end{equation*}%
\begin{equation*}
\dim _{H}\supp\mu \leq \overline{\dim }_{\Phi }\ \supp\mu \leq \overline{%
\dim }_{\Phi }\mu \leq \dim _{A}\mu .
\end{equation*}%
If $\mu $ is a doubling measure, then 
\begin{equation*}
\underline{\dim }_{\Phi }\mu \leq \underline{\dim }_{\Phi }\supp\mu .
\end{equation*}
\end{proposition}

The extremal properties of the Assouad dimensions are one reason for their
importance. Another is that a measure $\mu $ is doubling if and only if $%
\dim _{A}\mu <\infty $ (\cite{FH}) and uniformly perfect if and only if $%
\dim _{L}\mu >0$ (\cite{KL}).

\section{Our model for random Moran sets and measures}

\label{sec:construction}

In this paper we study random Borel measures generated by random weighted iterated function systems.  
The results in our paper are given in the context of subsets of $[0,1]$.   
This is mostly done for clarity and to avoid extra technical details.  
Extending to $\R^d$ is straightforward with the proper modification of the geometric separation condition.
We discuss extensions briefly in Section \ref{sec:extensions}.

\medskip

We begin with a probability space $(\Omega_0,P)$ and fix the integer $K\geq 2$ and
constants $A,B,\tau >0$ with $0 < K \cdot A  \le B < 1$ along with $\tau \le (1-B)/K$. 
By a \emph{random weighted iterated function system} (RwIFS) we mean a collection 
$\{\mathcal{I}_{\omega },\mathcal{P}_{\omega }\}_{\omega \in \Omega_0 }$ where 
$\mathcal{I}_{\omega}=\{S_{1}(\omega ),...,S_{K}(\omega )\}$ is a deterministic IFS consisting
of $K$ similarities $S_{j}(\omega )$ acting on $[0,1]$ and each 
$\mathcal{P}_{\omega }=\{p_{1}(\omega ),...,p_{K}(\omega )\}$ is a set of probability weights, 
that is $p_{i}(\omega )\geq 0$ for $i=1,...,K$ and $\sum_{i=1}^{K}p_{i}(\omega )=1$.

We denote the contraction (or scaling) factor of $S_{j}(\omega )$ by $a_{j}(\omega )$ and
assume 
\begin{equation*}
     A\leq a_{j}(\omega ) \le \sum_{i=1}^K a_i(\omega) \leq B.
\end{equation*}
The \emph{equicontractive} or \emph{homogeneous} case is when $a_j(\omega) = a_i(\omega)$ for all $i$ and $j$ and each $\omega \in \Omega_0$.

Furthermore, we suppose 
\begin{equation*}
    d(S_{i}(\omega )[0,1],\, S_{j}(\omega )[0,1])\geq \tau \text{ for all }i\neq j \text{ and all } \omega
\end{equation*}%
(which is possible by our choice of $A,B,\tau $), 
so each $\mathcal{I}_{\omega }$ is a strongly separated IFS on $[0,1]$. 
We also require that there is some $t>0$ such that 
${\mathbb E}_\omega(p_{i}^{-t}(\omega ))<\infty $ for each $i=1,...,K$. 
In particular, this ensures that the probability that any $p_{i}=0$ or $1$ is zero. 

Our construction of the random set and measure is driven by an iid sequence 
$({\mathcal I}_{\bomega_n}, {\mathcal P}_{\bomega_n})$, so this means that  
$\bomega \in \Omega_0^{\N} := \Omega$ is chosen according to the infinite product measure $\P$.
Specifically this means that all pairs $\{a_{i}(\bomega _{n}),p_{j}(\bomega _{n})\}$
are independent of $\{a_{k}(\bomega _{m}),p_{l}(\bomega _{m})\}$ for $n \ne m$.

For $\bomega \in \Omega $
we will put 
\begin{equation*}
I_{i_{1},...,i_{n}}(\bomega) = 
S_{i_{1}}(\bomega _{1})\circ \cdots \circ S_{i_{n}}(\bomega _{n})[0,1].
\end{equation*}%
This is a \textit{step (or level) }$n$\textit{\ Moran interval} of length $%
a_{i_{1}}(\bomega _{1}) \cdots a_{i_{n}}(\bomega _{n})$. We call the
step $n+1$ Moran intervals $I_{i_{1},..i_{n},j}(\bomega)$ for $%
j=1,...,K$ the \textit{children} of the step $n$ Moran interval $%
I_{i_{1},...,i_{n}}(\bomega)$ and note that the distance between
any two children of $I_{i_{1},...,i_{n}}(\bomega)$ is at least $%
\tau \left\vert I_{i_{1},...,i_{n}}(\bomega)\right\vert $. 

Corresponding to each $\bomega \in \Omega$
 is the \textit{random attractor } 
\begin{equation*}
     F_\bomega=\dbigcap\limits_{n}\dbigcup\limits_{i_{1},...,i_{n}\in
\{1,...,K\}}I_{i_{1},...,i_{n}}(\bomega),
\end{equation*}%
also known as a \textit{random Moran set}. 
When $\bomega \in \Omega$ is clear we may write $I_{n}(x)$ for the unique step $n$
Moran interval that contains $x\in F_{\bomega}$ and we observe that 
\begin{equation*}
A^{n}\leq \frac{\left\vert I_{N}(x)\right\vert }{\left\vert
I_{N+n}(x)\right\vert }\leq B^{n}\text{.}
\end{equation*}

Our \textit{random probability measure} $\mu _\bomega$ is given
by the rule $\mu _{\bomega}([0,1])=1$ and 
\begin{eqnarray*}
\mu _{\bomega}(I_{i_{1},...,i_{n}})(\bomega) &=&
  \mu _{\bomega}(I_{i_{1},...,i_{n-1}}(\bomega)) p_{i_{n}}(\bomega_n) \\
&=&p_{i_{1}}(\bomega _{1}) \cdots p_{i_{n}}(\bomega _{n}).
\end{eqnarray*}%
Almost surely, its support is the random set $F_{\mathbf{\omega }}.$

For those familiar with V-variable fractals (c.f. [2]), $F_{\mathbf{\omega }%
} $ is a 1-variable fractal and $\mu _{\mathbf{\omega }}$ a 1-variable
fractal measure.

We remark that $\F_{\bomega }$ has a \textquotedblleft uniform
separation\textquotedblright\ property in the sense that the distance
between any two children of $I_{N}$ is at least $\tau |I_{N}|$. 
Given the separation conditions on the construction, it is not hard to 
show that the $\Phi$ dimensions of both the random
set $\F_\bomega$ and the random measure $\mu_\bomega$ can 
be computed using only the lengths and measures of the Moran intervals (see \cite[Lemma 2]{HM2}).

In \cite{HM2} simple formulas are given for the almost sure small $\Phi $%
-dimensions of these measures.

\begin{theorem}[\protect{\cite[Theorem 21]{HM2}}]
\label{asdimmeasure} 
There is a set of full $\P$ measure $\Gamma \subseteq \Omega$ such that for all 
$\bomega \in \Gamma$ and all small
dimension functions $\Phi $ we have%
\begin{equation*}
  \dim_A \mu_\bomega = \overline{\dim }_{\Phi }\mu_\bomega=
 \max_{j=1,...,K}\left( \esssup_{\omega \in \Omega_0}
   \frac{\log p_{j}(\omega )}{\log a_{j}(\omega )}\right)
\end{equation*}
and
\begin{equation*}  
   \dim_L \mu_\bomega = \underline{\dim }_{\Phi }\mu_\bomega=
     \min_{j=1,...,K}\left( \essinf_{\omega \in \Omega_0}
    \frac{\log p_{j}(\omega )}{\log a_{j}(\omega )}\right).
\end{equation*}
\end{theorem}

For each $\omega \in \Omega _{0}$, let $d_{\omega }\geq 0$ satisfy 
\[
   \sum_{j=1}^{K}a_{j}(\omega )^{d_{\omega }}=1.
\]
Notice that $d_{\omega }$ is the Hausdorff (and upper and lower Assouad)
dimension of the self-similar Cantor subset of $[0,1]$ with contraction
ratios $a_j(\omega )$, $j=1,2,\ldots, K$.
The almost sure Hausdorff dimension of $F_\bomega$ is given by the unique $s$ which
satisfies $\ee_\omega \log\left( \sum_{j=1}^K a_j(\omega)^s \right) = 0$ (see, for instance, \cite{BHS}).

\medskip

\begin{proposition}  \label{prop:lowerphi}
Almost surely $\underline{\dim }_\Phi \mu _\bomega\leq \essinf_{\omega \in \Omega _{0}}d_{\omega }$ and $\esssup_{\omega \in \Omega _{0}}d_{\omega } \le \overline{\dim}_\Phi \mu_\bomega$ for all small dimension functions $\Phi$.
\end{proposition}

\begin{proof}
For each $\omega \in \Omega _{0}$, set 
\begin{equation*}
t_{j}(\omega )=\frac{\log p_{j}(\omega )}{\log a_{j}(\omega )}
\end{equation*}%
so that $p_{j}(\omega )=a_{j}(\omega )^{t_{j}(\omega )}$. Since $%
\sum_{j=1}^{K}a_{j}(\omega )^{d_{\omega }}=1=\sum_{j}p_{j}(\omega )$, there
cannot be any $\omega $ such that $t_{j}(\omega )>d_{\omega }$ or $t_j(\omega) < d_\omega$ for all $j$.
Therefore Theorem \ref{asdimmeasure} gives 
\[
\underline{\dim }_{\Phi }\mu _{{\boldsymbol{\omega }}}=
 \min_{j=1,...,K}\left( \essinf t_{j}(\omega )\right) \leq 
  \essinf d_{\omega}\text{ a.s. }
\]
and
\[
\overline{\dim }_{\Phi }\mu _\bomega=
 \max_{j=1,...,K}\left( \esssup t_{j}(\omega )\right) \ge 
  \esssup d_{\omega}\text{ a.s. }
\]
\end{proof}


\section{\protect\bigskip Dimension of random set and attainable measure dimensions}

\label{sec:SmallPhi}

Now we turn to our main results, which are split as to whether
the probability weights are allowed to depend on the scaling factors or not.
We also determine the almost sure small $\Phi$ dimension of the underlying random
set.

\subsection{The dependent probabilities case.} 
\label{sec:dependent}

We begin with the most general case when the probabilities may depend on the contraction
factors. Let 
\begin{equation} \label{eq:d_Delta}
D=\esssup_{\omega \in \Omega _{0}}d_{\omega } \quad \mbox{ and } \quad 
 d=\essinf_{\omega \in \Omega _{0}}d_{\omega }.  
\end{equation}%

\begin{theorem}
\label{thm:deprange}
With $d$ and $D$ defined as in (\ref{eq:d_Delta}), for any $r \in [D,\infty)$ (resp., $r \in (0,d]$), 
there is a choice of probabilities $p_i(\omega)$,
depending on the contraction factors $a_j(\omega)$ and strictly bounded away from zero, such that
almost surely
\[
   \overline{\dim}_\Phi \mu_\bomega = r \geq \overline{\dim}_\Phi F_\bomega \  (\mbox{resp.,} \  
   \underline{\dim}_\Phi \mu_\bomega = r \le \underline{\dim}_\Phi F_\bomega
 )
\] 
 simultaneously for all small dimension functions $\Phi$.
\end{theorem}

\begin{proof}
Let $r\in [D,\infty )$ and set $p_{1}(\omega )=a_{1}(\omega )^{r}$ for each $\omega \in \Omega_0$.
Almost surely over $\omega \in \Omega_0$, we have $p_{1}(\omega )\leq a_{1}(\omega )^{D}\leq a_{1}(\omega
)^{d_{\omega }}$. Thus%
\[
  p_{1}(\omega )+\sum_{j=2}^{K}a_{j}(\omega )^{d_{\omega }} \leq
     \sum_{j=1}^{K}a_{j}(\omega )^{d_{\omega }}=1\text{ a.s. and }
   p_{1}(\omega)+\sum_{j=2}^{K}a_{j}(\omega )^{0}>1,
\]
hence by continuity, for almost all $\omega \in \Omega_0$ there is some 
$r_{\omega }\leq d_{\omega }$ such that 
\[
    p_{1}(\omega )+\sum_{j=2}^{K}a_{j}(\omega )^{r_{\omega }}=1.
\]
Put $p_{j}(\omega )= a_{j}(\omega )^{r_{\omega }}$ for $j\neq 1$. 
As $\sum_{j=1}^{K}p_{j}(\omega )=1$, this defines probabilities 
$(p_{j}(\omega))_{j=1}^{K}$. 
Notice that $p_j(\omega) \ge \min\{ A, A^r \}$ since $a_j(\omega) \ge A$ and $r_\omega \le d_\omega \le 1$.
(For the measure zero set of $\omega $ where $r_{\omega }$
might not be defined just choose equal probabilities.)
 Because 
\[
   \max_{j=1,...,K}\left( \esssup \frac{\log p_{j}(\omega )}{\log a_{j}(\omega )}\, \right) 
 =\esssup \frac{\log p_{1}(\omega )}{\log a_{1}(\omega )}=r,
\]
we have $\overline{\dim }_{\Phi }\mu_\bomega=r$ a.s. according to Theorem \ref{asdimmeasure}.

The symmetric argument proves the result for the lower $\Phi $ dimensions.
For this we need to note that since the chosen probabilities are uniformly bounded away from zero, the resulting measure is 
always doubling and hence $\underline{\dim}_\Phi \mu_\bomega \le \underline{\dim}_\Phi F_\bomega$.
\end{proof}

\begin{remark}
Whether or not $0$ or $\infty $ can be achieved as dimensions of measures
depends on the underlying probability space. See Remark \ref{ZeroInf} for
more discussion on this.
\end{remark}

\subsection{Almost sure dimension of the random set}
\label{subsec:smallrandomset}

Our argument for obtaining the lower bound in the proof of Theorem \ref{thm:asdimset} is motivated by the argument originally from \cite{Ol}
and then modified in \cite[Section 6.2]{FMT}.
We first need to set up notation so that we can state a useful combinatorial estimate (inspired by \cite{Ol}).

\smallskip

Let $\Lambda = \{1,2,\ldots, K\}$ and ${\mathcal T} = {\displaystyle \bigcup_{i =0}^n } \Lambda^i$ be a $K$-ary tree of depth $n$.
For each $1 \le j \le K$ and $1 \le \ell \le n$, let $c_{j,\ell} \in [A,B] \subset (0,1)$ with $\sum_{j=1}^K c_{j,\ell} \le B$ for all $\ell$.

For each node $\sigma = \sigma_1 \sigma_2 \ldots \sigma_m \in \Lambda^m \subset {\mathcal T}$, we define:
\begin{enumerate}
   \item[(i)] $\hat{\sigma} := \sigma_1 \sigma_2 \ldots \sigma_{m-1}$ as the \emph{predecessor} of $\sigma$ and 

   \item[(ii)] $c_\sigma := c_{\sigma_1,1} \cdot c_{\sigma_2,2} \cdot c_{\sigma_3,3} \cdots c_{\sigma_m,m}$ as the label on the node $\sigma \in {\mathcal T}$.

\end{enumerate}
We omit the proof of this next lemma, as it is similar to the proof of Lemma 2.2 in \cite{Ol}.

\begin{lemma}[Geometric Lemma] \label{lem:geometric}
 Suppose that ${\mathcal T}$ is a labeled $K$-ary tree of depth $n$ (as above).
Let $0 < r < R$ be such that $r \ge c_\sigma R$ whenever $|\sigma| = n$.
Suppose that for some $s \ge 0$ we have $\displaystyle \sum_{\ell=1}^K c_{j,\ell}^s \ge 1$ for all $\ell$ and
let 
\[
       \Upsilon_r = \{ \sigma \in {\mathcal T}:   c_\sigma \le \frac{2 }{\tau A} \frac{r}{R} < c_{\hat{\sigma}}\}.
\]
Then $|\Upsilon_r| \ge C (R/r)^s$ with $C = (\tau A/2)^s > 0$.
\end{lemma}

\begin{theorem} \label{thm:asdimset} 
There is a set $\Gamma \subset \Omega$ of full measure, such that $\overline{\dim}_\Phi F_\bomega = D$
 and $\underline{\dim}_\Phi F_\bomega = d$ for all $\bomega \in \Gamma$ and for all small
dimension functions $\Phi$.
\end{theorem}

\begin{proof}
We show the result for the upper dimension as the one for the lower is similar.

\subsubsection*{Proof of the upper bound $\overline{\dim}_\Phi F_\bomega \le D$} \quad
The upper bound is a simple consequence of Theorem \ref{thm:deprange} as we can choose probabilities
so that $D = \overline{\dim}_\Phi \mu_\bomega \geq  \overline{\dim}_\Phi F_\bomega$.

\subsubsection*{Proof of the lower bound $\overline{\dim}_\Phi F \ge D$}\quad

\smallskip

From the definition of $D$, for each $i \in \N$ there is some $0 < \eta_i < 1$ such that
\[
    P( \omega \in \Omega_0 :  \sum_{j=1}^K a_j^{D - 1/(2i)}(\omega)  > 1 ) \ge \eta_i.
\]
Let
\[
    J_i = \frac{\vert \log B\vert}{K \vert \log \eta_i \vert} \quad \mbox{ and } \quad \Phi_i(t) = \frac{ J_i \log\vert \log t \vert}{\vert \log t \vert}.
\]
For each positive integer $N$, let
\[
    \chi_{N,i} = \frac{ J_i \log( N \vert \log A \vert)}{\vert \log B \vert}.
\]
Clearly, $\chi_{N,i} \to \infty$ as $N \to \infty$ and the definitions ensure that $\eta_i^{\chi_{N,i}} \ge 1/N$ for large enough $N$.  Set
\[
    \Gamma_{N,i} = \{ \bomega \in \Omega : \sum_{j=1}^K a_j(\bomega_\ell)^{D - 1/(2 i)} \ge 1  \mbox{ for } \ell = N+1, \ldots, N+\chi_{N,i} \}.
\]
As the sets of scaling factors $(a_1(\bomega_n), a_2(\bomega_n), \ldots, a_K(\bomega_n) )$ are independent for different $n$,
\[
   \P(\Gamma_{N,i}) = \prod_{\ell=N+1}^{N+\chi_{N,i}} \P( \bomega \in \Omega : \sum_{j=1}^K a_j(\bomega_\ell)^{D - 1/(2i)} \ge 1 ) \ge \eta_i^{\chi_{N,i}} \ge 1/N
\]
for all large enough $N$.
Thus if we let $N_k = k \log k$, then for some suitable large $K_0$, 
\[
    \sum_k \P(\Gamma_{N_k,i}) \ge \sum_{k \ge K_0} \frac{1}{k \log k} = \infty.
\]
As we can replace $\eta_i$ with any smaller, strictly positive number, there is no loss in generality in assuming it is so small so that $N_{k+1} > N_k + \chi_{N_k,i}$.
Hence, the events $\Gamma_{N_k,i}$ are independent and thus the Borel-Cantelli lemma implies that $\P( \Gamma_{N_k,i} \, i.o.) = 1$ for each (fixed) $i$.
Let $\Gamma_i$ be this set of full measure.

Take any $\bomega \in \Gamma_i$ and consider any Moran interval, $I_{N}(\bomega)$, of step $N := N_k$.
Let $I_n(\bomega)$ be the largest descendant of  $I_{N}$ at level $n = N + \chi_{N,i}$.
Since $|I_{N}| \ge A^{N}$ and the function $t^{\Phi_i(t)}$ decreases as $t$ decreases to $0$, the choice of $\chi_{N,i}$ ensures
\[
    |I_N|^{\Phi_i(|I_N|)} \ge A^{N \Phi_i(A^N)} = A^{\frac{J_i \log(N \vert \log A \vert)}{\vert \log A \vert}} = B^{\chi_{N,i}} \ge \frac{\vert I_n \vert}{\vert I_N \vert}.
\]
Hence $|I_n| \le |I_N|^{1 + \Phi_i(|I_N|)}$.

The idea is that now we use Lemma \ref{lem:geometric} on the part of the ``tree'' used in the construction of $F_\bomega$ which is ``below'' $I_N$,
on the levels between $N$ and $n = N + \chi_{N,i}$.
To this end, let $R = |I_N|$ and $r = |I_n|$ and $c_{j,\ell} = a_j(\bomega_{N-1+\ell})$ for $j=1,2,\ldots, K$ and $\ell =1,2,\ldots, \chi_{N,i}$.
Notice that by our choice of $N = N_k$, we have $\sum_j c_{j,\ell}^{D-1/(2i)} \ge 1$ for each $\ell=1,2,\ldots, \chi_{N,i}$.
Furthermore, since $I_n$ is the largest descendant of $I_N$, we have $r = |I_n| \ge c_\sigma |I_N| = c_\sigma R$ for all $\sigma$ with $|\sigma|=\chi_{N,i}$.

We now apply Lemma \ref{lem:geometric}  with $s = D - 1/(2 i)$ and obtain that
\[
   |\Upsilon_r| \ge C (R/r)^{D - 1/(2i)} = C (|I_N|/|I_n|)^{D-1/(2i)}.
\]
For any $\sigma \in \Upsilon_r$, we have $2 r < \tau R A c_{\hat{\sigma}} \le \tau |I_N| c_\sigma$ and thus the separation between the children of $I_{N+|\sigma|}$
is at least $2 r$ and so any $r$-covering of $I_N$ must have at least as many sets as there are children of the $I_{N+|\sigma|}$ for $\sigma \in \Upsilon_r$.
Thus 
\[
   \cN_r(I_N) \ge |\Upsilon_r| \ge C (|I_N|/|I_n|)^{D - 1/(2i)}.
\]

Since $\bomega \in \Gamma_i$, it follows that there is a sequence of integers $N_i \to \infty$ and $r_i \to 0$ with $r_i \le |I_{N_i}|^{1 + \Phi_i(|I_{N_i}|)}$
and 
\[
   \cN_{r_i}(I_{N_i}) \ge  C (|I_{N_i}|/r_i)^{D - 1/(2i)}.
\]
This means that $\overline{\dim}_{\Phi_i} F_\bomega \ge D - 1/i$ for all $\bomega \in \Gamma_i$.

Now let $\Gamma = \bigcap_{i=1}^{\infty} \Gamma_i$, a set of full measure, and assume that $\Phi$ is any small dimension function.  
Then there is some function $H(t) \to 0$ as $t \to 0$ so that
\[
   \Phi(t) \le \frac{H(t) \log\vert \log t \vert}{\vert \log t \vert} \mbox{ for all } t \le t_0.
\]
Consequently, for each $i$ there is some $t_i > 0$ such that $\Phi(t) \le \Phi_i(t)$ for all $t \le t_i$.
This property and our observations above ensure that
\[
   \overline{\dim}_\Phi F_\bomega \ge \overline{\dim}_{\Phi_i} F_\bomega \ge D - 1/i
\]
 for all $i$ and
$\bomega \in \Gamma$.  We conclude that $\overline{\dim}_\Phi F_\bomega \ge D$ for all $\bomega \in \Gamma$, as we desired to show.
\end{proof}

Since the  Assouad dimensions are  small $\Phi$ dimensions, an immediate corollary is the following.

\begin{corollary}
  Almost surely we have that $\dim_A F_\bomega = D$ and $\dim_L F_\bomega = d$.
\end{corollary}

Another corollary is the relationship between the almost sure value of $\underline{\dim}_\Phi \mu_\bomega$ and $\underline{\dim}_\Phi F_\bomega$, completing the relationship started in Proposition \ref{prop:lowerphi}.

\begin{corollary}
Almost surely we have $\underline{\dim}_\Phi \mu_\bomega \le \underline{\dim}_\Phi F_\bomega$.
\end{corollary}

We emphasize that $\underline{\dim}_\Phi \nu \le \supp \nu$ is always true whenever $\nu$ is a doubling measure, but it is possible
that $\overline{\dim}_A \mu_\bomega = \infty$, in which case $\mu_\bomega$ is almost surely not doubling, see Remark \ref{ZeroInf}.

\subsection{The independent probabilities case}

Now we turn again to the problem of finding the range of the attainable a.s.\ measure dimensions.
A natural and particularly interesting example of our setup is when $\Omega_0$ is a product of
probability spaces, $\Gamma _{1}\times \Gamma _{2}$, with the product
measure. Let $\omega =(u,v)$ with $u\in \Gamma _{1}$ and $v\in \Gamma _{2},$
and suppose that $\mathcal{I}_{\omega }$ depends only on $u$ and
$\mathcal{P}_{\omega }$ only on $v$. 
Then the contraction factors, which depend only on $u$, 
and the probabilities which depend only on $v$, are independent of each
other. We call this the\textit{\ independent probabilities} case.

This situation turns out to be very interesting and quite different from the general case in 
Section \ref{sec:dependent} where the probabilities are allowed to depend on the scaling factors.
In this case, we will see that typically the smallest attainable almost sure upper $\Phi $ dimension 
of the random measures is strictly greater than the upper Assouad dimension of the random set for small
dimension functions $\Phi$; that is, there is a ``gap'' between the attainable measure dimensions
and the dimension of the underlying set.
A similar gap also typically exists for the lower $\Phi$ dimensions where the largest attainable almost sure value is strictly
smaller than the lower Assouad dimension of the random set.

Given contraction factors $a_i(\omega)$ for $i=1,2,\ldots, K$ and 
 independently chosen probabilities $p_j(\omega)$ for $j=1,...K$, we set 
\begin{align*}
   \overline{a}_j &=\esssup a_{j}(\omega ), &  \overline{p}_j &=\esssup p_{j}(\omega ),\\
   \underline{a}_j &=\essinf a_{j}(\omega ), &  \underline{p}_j &=\essinf p_{j}(\omega ),\\
\end{align*}
Note that in this situation the formulas from Theorem \ref{asdimmeasure} become
\begin{equation*}
\overline{\dim }_{\Phi }\mu_\bomega=
 \max_{j=1,...,K}\left( \frac{\log \underline{p}_{j}}{\log \overline{a}_{j}}\text{ }\right) \text{ a.s.}
\end{equation*}%
and 
\begin{equation*}
\underline{\dim }_{\Phi }\mu_\bomega=
\min_{j=1,...,K}\left( \frac{\log \overline{p}_{j}}{\log \underline{a}_{j}}\text{ }\right) \text{ a.s.}
\end{equation*}

\begin{theorem}\label{thm:asindrange}
Let $\Delta \ge 0$  be defined by
\[
    \sum_{j=1}^K \overline{a}_j^\Delta = 1.
\]
Then $\overline{\dim}_\Phi \mu_\bomega \ge \Delta$ almost surely whenever the probabilities
 are chosen independently from the scaling factors.
Furthermore, given any $R \in [\Delta, \infty)$, there is a choice of independent probabilities 
such that $\overline{\dim}_\Phi \mu_\bomega = R$.

Similarly, if $\delta$ satisfies 
\[
  \sum_j \underline{a}_j^\delta = 1,
\] 
then $\underline{\dim}_\Phi \mu_\bomega \le \delta$ almost
surely whenever the probabilities are chosen independently.
Further, for all $r \in (0,\delta]$ there is a choice of independent probabilities
so that $\underline{\dim}_\Phi \mu_\bomega = r$ almost surely.
\end{theorem}

\begin{proof}
Without loss of generality, assume 
\begin{equation*}
  \overline{\dim }_{\Phi }\mu _{\bomega }=\frac{\log \underline{p}_{1}}{\log \overline{a}_{1}}
     \quad  \text{a.s.}
\end{equation*}
Define $s_{j}$ by $\overline{a}_{j}=\overline{a}_{1}^{s_{j}}$. 
Since
\[
    \frac{\log \underline{p}_{j}}{\log \overline{a}_{j}}\leq \frac{\log \underline{p}_{1}}{\log \overline{a}_{1}},
\]
$\underline{p}_{j}\geq \underline{p}_{1}^{s_{j}}$. 
Also, $\sum_{j=1}^{K} \underline{p}_{j}\leq 1$ since $%
\underline{p}_{j}\leq p_{j}(\omega )$ a.s. and $\sum_{j}p_{j}(\omega )=1$ for all $%
\omega \in \Omega_0$. 
Hence
\begin{equation*}
\sum_{j=1}^{K} \underline{p}_{1}^{s_{j}}\leq \sum_{j} \underline{p}_{j}\leq
1=\sum_{j} \overline{a}_{j}^{\Delta}=\sum_{j} \overline{a}_{1}^{s_{j}\Delta}.
\end{equation*}%
That implies $\underline{p}_{1}\leq \overline{a}_{1}^{\Delta}$ and therefore 
$\Delta\leq \log \underline{p}_{1}/\log \overline{a}_{1} $. 
Thus $\Delta$ is a lower bound on the a.s. upper $\Phi $ dimension of
the random measures.

To see that each $r\in [\Delta,\infty)$ arises as the a.s. dimension of a
random measure simply take $p_1=\overline{a}_{1}^{r}$. 
Since $r\geq \Delta,$ there must be some $r_{j} \leq \Delta$ such that 
\[
   \underline{p}_{1}+\sum_{j=2}^{K} \overline{a}_{j}^{r_{j}}=1.
\]
Setting $p_{j}=\overline{a}_{j}^{r_{j}}$ for $j=2,3,\ldots, K$ completes the definition 
of a (constant) set of probabilities that we can choose
independently of the $a_i(\omega)$.
Clearly 
\[
   \max_{j=1,...,K}\left( \frac{\log p_j}{\log \overline{a}_{j}}\, \right) =r.
\]

The argument for the lower $\Phi $ dimension is symmetric.
\end{proof}

\begin{remark}
Since $\overline{a}_{j}\geq a_{j}(\omega )$ a.s., it is clear that almost surely 
$\Delta\geq\dim _{A}\F_{\bomega }=D$ and typically is strictly greater. 
For example, if the underlying probability space is finite, then $\Delta=D$ if and only if there
is some $\omega \in \Omega_0$ such that $\overline{a}_{j}=a_{j}(\omega)$ for each $j=1,...,K$.

As a simple example, suppose we use only the two different pairs of contraction factors 
$(a_1 = 1/2, a_2 = 1/4)$ and $(b_1= 1/4, b_2= 1/2)$ and choose from these equally likely.
Then almost surely
\[
    d  = D = \dim_A F_\bomega = \dim_L F_\bomega = \dim_H F_\bomega = \frac{\ln( \sqrt{5}/2 + 1/2)}{\ln(2)} \approx 0.6942,
\]
which is the positive solution to $1/2^x + 1/4^x = 1$ and is $d_\omega$ for both sets of contraction factors.
However, the minimal attainable almost sure upper measure dimension is $\Delta = 1$ (the solution to $1/2^x + 1/2^x = 1$)
and the maximal attainable almost sure lower measure dimension is $\delta = 1/2$ (the solution to $1/4^x + 1/4^x=1$).

\smallskip
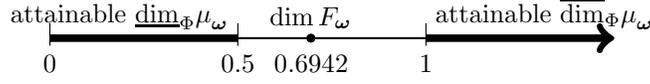
\begin{center}
\begin{tikzpicture}[scale=0.5]

      \draw[thick]  (0,0) -- (15,0);
      \draw[fill=black] (6.942,0) circle (3 pt); \node[above] at (6.942,0) {$\dim F_\bomega$};
      \node[below] at (6.942,-0.2) {0.6942};

      \draw[black] (10,-0.2) -- (10,0.4); \node[above right] at (10,0) {attainable $\overline{\dim}_\Phi \mu_\bomega$};
      \draw[line width=1mm,->] (10,0) -- (15,0);
      \node[below] at (10,-0.2)  {1};

      \draw[black] (5,-0.2) -- (5,0.4); \node[above left] at (5,0) {attainable $\underline{\dim}_\Phi \mu_\bomega$};
      \draw[line width=1mm] (0,0) -- (5,0);
      \node[below] at (5,-0.2) {0.5};
      \draw[black] (0,-0.2) -- (0,0.4);  \node[below] at (0,-0.2) {$0$};

\end{tikzpicture}
\end{center}
\smallskip
This setup corresponds to choosing randomly between the two IFSes
\[
    \{ w_0(x) = x/2, w_1(x) = x/4+3/4 \} \quad \mbox{ and } \quad \{ w_0(x) = x/4, w_1(x) = x/2 + 1/2 \}
\]
and using the same probabilities $p$ and $1-p$ for both.

Notice that if we allow the values of $a_{1}$ and $b_{2}$ to tend to one rather than each being $1/2$
(which can be done while sending $a_2$ and $b_1$ to zero), then we can arrange for $\Delta$ to tend to infinity, while,
of course, the Assouad dimension of any set in $\mathbb{R}$ is never more
than $1$.
This means that the ``gap'' between the attainable values of the upper measure dimension and the upper dimension of the underlying set
can be arbitrarily large.

\end{remark}

\medskip

\begin{remark}
\label{ZeroInf}Whether or not the value $r=\infty $ (or $r=0$) can be
achieved depends on the underlying probability space. Since the contraction
factors are bounded away from zero (and one), in the independent probability
case we can only obtain the value infinity (or zero) as an upper (resp.,
lower) small $\Phi $ dimension if it is possible for some $\underline{p}_{k}=0$,
 (resp., $\overline{p}_{k}=1$). 
Simultaneously, for each $j$ there would need to be some $t>0$
such that $\mathbb{E(}p_{j}^{-t})<\infty$. 
This is impossible to achieve when the probability space is finite, but is possible, 
for example, if the underlying probability space is $[0,1]$ with Lebesgue measure. 
A similar statement can be made in the case of dependent probabilities.
\end{remark}

\begin{remark}
As $\sum_{j} \underline{a}_{j}<1,$ we know that $\delta $ is the dimension 
(Assouad and Hausdorff) of any strongly separated Moran set which has 
$\underline{a}_{j}$ as its contraction factors. 
Similarly, in the case when $\sum_{j=}^{K} \overline{a}_{j}<1,$ then $\Delta$ is the dimension 
of any strongly separated Moran set with $\overline{a}_{j}$ as contraction factors. 
However, in general, such Moran sets are not any of the possible random
Moran sets in the construction, so it is unclear if there is a geometric
reason why they might be relevant. 
Of course, $\sum \overline{a}_{j}\geq 1$ is also possible. 
In this case there is still an associated self-similar set with these contraction factors, but not satisfying the open set condition. 
$\Delta$ is not the dimension of this set, in general. 
\end{remark}

\section{Extensions of these results}
\label{sec:extensions}

We assumed in our setup in Section \ref{sec:construction} that the number of similarities, 
$K,$ was the same for each $\mathcal{I}_{\omega }$. 
This is only a convenience; all the results go through with the obvious notational modifications
provided we only assume the number of similarities are bounded by $K$.
(Since we require that the contraction factors are bounded away from $0$ and
there is a gap of size at least $\tau $ between step one Moran intervals,
boundedness is necessary.)

We also assumed that the similarities acted on $[0,1]$. Again this was only
a convenience - any compact subset of $\mathbb{R}$, or even $\mathbb{R}^{d},$
could replace $[0,1]$. Of course, in $\mathbb{R}^{d}$ for $d\geq 2$ the
Moran sets are not intervals and the contraction factors determine the
diameters of the Moran sets rather than their lengths.

We do not even require the sets $\mathcal{I}_{\omega }$ to be collections of
similarities. It is enough to define 
$\mathcal{I}_{\omega }=\{a_{1}(\omega),...,a_{K}(\omega )\}$,
 with the same requirements on the random variables $a_{j}(\omega )$ 
as we had before for the contraction factors. Give $\omega\in \Omega $, 
we define the random Moran set ${F}_\bomega$
inductively beginning with the closed interval $I_{0}=[0,1]$. At the first
step, we choose $K$ closed subintervals  
$I_{1}^{(i)}(\omega _{1})\subset I_{0},$ $i=1,...,K$, with 
$|I_{1}^{(i)}|=a_{i}(\bomega_1)$
and $d(I_{1}^{(i)},I_{1}^{(j)})\geq \tau $ for $i\neq j$. 
These are the step  one Moran intervals, the children of $I_{0}$. 
We denote by $F_\bomega^{(1)}$  the union of the $K$ step one Moran
intervals.  
Having inductively created the $K^{n-1}$ Moran intervals of
step $n-1$, we choose $K$ closed subintervals of each such interval $I$ of
sizes $a_{i}(\bomega_{n})|I|$ for $i=1,...,K$,
 with the separation between any two of the children of $I$ being at least 
$\tau |I|$. 
We remark that this is all we require of the children. 
In particular, the relative sizes of corresponding children of two different level $n$ Moran
intervals have to be the same, but not the spacing between them. 
This produces $K^{n}$ Moran intervals of step $n$ whose union we denote by ${F}_{\bomega }^{(n)}$,
a closed subset of $[0,1]$. The random Moran set\textbf{\ }${F}_{{%
\boldsymbol{\omega }}}$ is the compact set 
\[
   {\F}_\bomega=\bigcap_{n=1}^{\infty }{F}_\bomega^{(n)}. 
\]
The random measure $\mu _{\bomega }$ is defined in the same manner as before.
All the theorems hold in this more general setting, with the same proofs.
Again we have the two cases of probabilities which are either dependent or independent
of the scaling factors.

\end{document}